\documentclass{amsart}

\usepackage{all2021}
\usepackage{graphics}
\usepackage{xypic}
\usepackage{xcolor}
\usepackage[style=alphabetic]{biblatex}
\begin{document}

\newcommand{\Sh}{\operatorname{Sh}}
\newcommand{\LM}{\operatorname{LM}}
\renewcommand{\Vec}{\mathbf{Vec}}
\newcommand{\FI}{\mathbf{FI}}
\newcommand{\FIop}{\mathbf{FI^{op}}}
\newcommand{\bC}{\mathbf{PM}}
\newcommand*\closure[1]{\overline{#1}}

\title{Image closure of symmetric wide-matrix varieties}
\author{Jan Draisma, Rob H.~Eggermont, Azhar Farooq, and Leandro Meier}
\thanks{JD was partially supported by Vici grant
639.033.514 from the Netherlands Organisation for Scientific
Research (NWO) and Project Grant 200021\_191981 from the
Swiss National Science Foundation (SNF).
AF was supported by Vici grant 639.033.514.
RE was supported by Veni grant 016.Veni.192.113 from NWO}

\maketitle

\begin{abstract}
Let $X$ be an affine scheme of $k\times \NN$-matrices and
$Y$ be an affine scheme of $\NN \times \cdots \times
\NN$-dimensional tensors. The group $\Sym(\NN)$ acts
naturally on both $X$ and $Y$ and on their coordinate rings.
We show that the Zariski closure of the image of a
$\Sym(\NN)$-equivariant morphism of schemes from $X$ to $Y$ is defined by finitely many $\Sym(\NN)$-orbits in the coordinate ring of $Y$. Moreover, we prove that the closure of the image of this map is $\Sym(\NN)$-Noetherian, that is, every descending chain of $\Sym(\NN)$-stable closed subsets stabilizes.
\end{abstract}

\section{Introduction}

In recent years, many finiteness properties of rings with countably
many variables and infinite dimensional varieties have been studied,
especially of those acted upon by the infinite symmetric group. A
number of fundamental results have been established, but still many have
yet to be discovered.

Most of these results build on the theorem from
\cite{cohen:metabelian,cohen:closure}, rediscovered in
\cite{aschenbrenner-hillar:finitesymmetric,hillar-sullivant:GBinfdim}, that
the polynomial ring in $k \times \NN$ variables is $\Sym(\NN)$-Noetherian
with the natural action of the symmetric group $\Sym(\NN)$. This implies
that the space of $k\times \NN$-matrices equipped with the induced
action of $\Sym(\NN)$ is topologically $\Sym(\NN)$-Noetherian. However,
already the space of $\NN \times \NN$-matrices equipped with the natural
action of $\Sym (\NN)$, or even with the action of $\Sym(\NN) \times
\Sym(\NN)$, is not equivariantly Noetherian, and it gets only worse for
higher-dimensional tensors. This is problematic, because many varieties
of relevance to applications, such as the $k$-factor model
\cite{draisma:kfactor} and hierarchical models
\cite{hillar-sullivant:GBinfdim} naturally live in matrix or tensor spaces and
are preserved by (copies of) the symmetric group.

Hence it is interesting to find $\Sym(\NN)$-stable subvarieties of
$\NN \times \NN$-matrices, or $\NN \times \cdots \times \NN$-tensors,
that are defined by finitely many orbits of equations and equivariantly
Noetherian. In this paper we study such subvarieties that
arise as image closures of $\Sym(\NN)$-equivariant maps from the space
of $k \times \NN$-matrices, and we show that these are always defined,
set-theoretically, by finitely many $\Sym(\NN)$-orbits of equations,
as well as themselves topologically $\Sym(\NN)$-Noetherian.

\subsection{Main Theorem}\label{ssec:main theorem}

All algebras and schemes are defined over a Noetherian ring
$K$. All schemes are \textit{affine schemes} and we will usually
drop the adjective \textit{affine}. Let $\FI$ denotes the category
of finite sets with injective maps. An important insight from
\cite{church-ellenberg-farb:FImodules} is that functors from $\FI$
are the right language for studying structures with
increasing numbers of variables acted upon by increasing
symmetric groups.

An $\FI$-algebra over $K$ is a covariant functor from $\FI$ to the
category of algebras over $K$.  An $\FI$-ideal $I$
of an $\FI$-algebra $A$ assigns to each finite set $S$ an ideal $I(S)$
of $A(S)$, in such a manner that the algebra homomorphism $A(\pi):A(S)
\to A(T)$ corresponding to any $\pi \in \Hom_\FI(S,T)$ maps $I(S)$
into $I(T)$. An $\FIop$-scheme over $K$ is a contravariant functor from $\FI$ to the
category of schemes over $K$. A closed subscheme $Z$ of an $\FIop$-scheme
$X$ is a subfunctor of $X$ such that, for each finite set $S$, $Z(S)$
is a closed subscheme of $X(S)$. Morphisms of $\FI$-algebras are natural
transformations of functors. Morphisms of $\FIop$-schemes are defined
dually.

We say that an $\FI$-algebra $A$ is {\em finitely generated in width at most $1$}
if $A$ is generated by $A(\emptyset)$ and $A(\{1\})$, and
both of these are finitely generated. This is equivalent
to the statement that $A$ can be realised as a quotient (in the category
of $\FI$-algebras over $K$) of an algebra of the form
\[ S \mapsto K[x_{ij} \mid i \in [k], j \in S]; \]
see \cite[Lemma 2.6.8]{draisma-eggermont-farooq:components}.
Geometetrically, $\Spec(A)$ is then an $\FIop$-subscheme over $K$ of
the scheme that sends $S$ to $k \times S$-matrices.

\begin{thm}[Main Theorem]\label{thm:main theorem}
Let $\phi$ be a morphism of $\FI$-algebras $B \rightarrow
A$, where $A$ and $B$ both are finitely generated
$\FI$-algebras over a Noetherian ring $K$. Suppose,
moreover, that $A$ is generated in width at most $1$. Then
the Zariski closure $\closure {\Im (\phi^*)}$ of the image
of $\Spec(A)$ under the dual map $\phi^*$ is defined, as a
set, by finitely many elements in $B$. Equivalently, there exist a finitely generated $\FI$-ideal $I$ in $B$ such that the radical of $I$ is equal to the radical of the kernel $\Ker \phi$.\\
Moreover, $\closure {\Im (\phi^*)}$ is topologically
Noetherian, that is, it satisfies the descending chain
condition on reduced closed subschemes.
\end{thm}

An immediate consequence of the Main Theorem is the
following.

\begin{cor}
For any $\Sym(\NN)$-equivariant morphism $\phi$ of $K$-schemes from the
scheme of $k \times \NN$-matrices to the scheme of $\NN \times \cdots
\times \NN$-tensors, the image closure
$\overline{\Im(\phi)}$ of $\phi$ is defined,
set-theoretically, by finitely many $\Sym(\NN)$-orbits of
equations. Furthermore, it is topologically $\Sym(\NN)$-Noetherian:
any descending chain
\[ \overline{\im(\phi)} \supseteq X_1 \supseteq X_2
\supseteq \ldots \]
of reduced, closed subschemes stabilises. \hfill $\square$
\end{cor}

We do not know whether $\ker \phi$ itself is necessarily finitely
generated.

\begin{conj}\label{kernel}
Let $\phi$ be the $\FI$-algebra homomorphism defined in the Main Theorem \ref{thm:main theorem}. Then the $\FI$-ideal $\Ker \phi$ is finitely generated.
\end{conj}

A particular case of the above conjecture was stated in \cite[][Conjecture 5.10]{aschenbrenner-hillar:finitesymmetric} and proved in \cite{DEKL} when generators of the $\FI$-algebra $B$ are mapped under $\phi$ to monomials in $A$.

On the other hand, in the setting of the Main Theorem, there are
examples where $\closure{\Im(\phi^*)}$ is not scheme-theoretically
Noetherian, i.e., has an infinite strictly descending chain of closed
(but non-reduced) $\FIop$-subschemes: in \cite[Theorem 43]{Draisma23a}
it is proved that the cone over the Grassmannian of $2 \times S$-spaces,
with $S \in \FI$, has this bad behaviour in characteristic $2$
(but not in characteristic zero due to the existence of
Reynolds operators, see \cite{draisma:kfactor}).

\subsection{Relations to existing literature}

Motivation for this work comes from finiteness questions about chains of
$\Sym$-invariant ideals that arise in commutative algebra and adjacent
areas. These questions concern the behaviour of homological invariants
\cite{Betti_numbers_of_symmetric_shifted_ideals,Le20,Le21,vanle2022regularity},
bivariate Hilbert series recording both degree and
number of variables \cite{Krone17,NAGEL2017204},
and irreducible components of $\FIop$-varieties
\cite{Nagpal20,Nagpal21,draisma-eggermont-farooq:components}.  Some of
these developments were inspired by challenging questions in algebraic statistics
\cite{aschenbrenner-hillar:finitesymmetric,Drton_2006,draisma:kfactor,
Brouwer11}, and similar techniques have been applied in
discrete geometry \cite{Kahle22} and real-algebraic geometry
\cite{Kummer22}. The most efficient framework, in our opinion, to
study such questions is that of $\FI$-algebras and $\FIop$-schemes,
\cite{Nagel-Romer,draisma-eggermont-farooq:components}.

To illustrate the strength of our Main Theorem, we explain
how it immediately implies the main result in
\cite{draisma:kfactor}.

\begin{ex}
In \cite{Drton_2006} the question was raised whether the ideal of
polynomial relations among the entries of a matrix of the form $\Sigma=A
A^T + D$, with $A$ running through $\RR^{n \times k}$ and
$D$ running through the (positive-definite)
diagonal matrices, is generated by a finite number of $\Sym([n])$-orbits
for $n \to \infty$. The set of all such matrices is the {\em Gaussian
$k$-factor model}. This question remains open to this date. However,
in \cite{Brouwer11} an affirmative answer was given for $k=2$, and in
\cite{draisma:kfactor} an affirmative answer for a set-theoretical version
was established for arbitrary $k$. The latter also follows from the
Main Theorem, because the map that sends the matrix entry $\sigma_{ij}$
to the entry $\sum_{l=1}^k a_{il}a_{jl} + \delta_{ij} \cdot d_{ii}$ is
a homomorphism from the $\FI$-algebra $S \mapsto \RR[\sigma_{ij} \mid
i,j \in S]$ to the $\FI$-algebra $S \mapsto \RR[a_{il},d_{ii} \mid l \in
[k], i \in S]$, and the latter is generated in width $\leq
1$.
\end{ex}

In the example, as in the rest of the paper, $[k]$ stands for
$\{1,\ldots,k\}$, and in particular $[0]=\emptyset$.

\subsection{Organisation of this paper}

In Section 2 we recall elementary properties of finitely generated $\FI$-algebras. Section 3 contains basic results about topological Noetherianity. Section 4 is devoted to the reduction of the Main Theorem \ref{thm:main theorem} to the case of free $\FI$-algebras. In Section 5 we prove our Main Theorem. To this end we prove that image of the map $\phi^*$ lies in the closed subscheme of $\Spec(B)$ defined by off-diagonal $(l+1)\times (l+1)$-subdeterminants of flattnings, see Lemma \ref{lem:containment_in_a_subdeterminant_variety}. This subscheme is defined by finitely many equations (Lemma \ref{lem:Z_l is f.gen.}) and is contained in the image of a topologically Noetherian space, see Proposition \ref{prop:induction} and Lemma \ref{lem:Z}.

A result of potentially independent interest is the following tensor
completion result: a tensor labelled by the tuples in $S \times \cdots \times
S$ in which no entry appears more than once (an {\em off-diagonal tensor})
and of which all the meaningful $(l+1) \times (l+1)$-subdeterminants of
flattenings are zero, can be completed to a full tensor (including the
diagonal) of some bounded rank; see \ref{prop:induction}.

\section{Preliminaries}

In this section we recall definitions and fundamental results about $\FI$-algebras and $\FIop$-schemes. For more details about them see \cite{Nagel-Romer,draisma-eggermont-farooq:components}.

\subsection{$\FI$-algebras and $\FIop$-schemes} \label{ssec:FISchemes}

The following notion was introduced in
\cite{church-ellenberg-farb:FImodules}.

\begin{de}
The category $\FI$ (for {\em Finite sets with Injections})  has as objects finite sets and for finite sets $S$ and $T$
the set $\Hom_{\FI}(S,T)$ is the set of injections from $S$ to $T$. $\FIop$ denotes its opposite category.
\end{de}

Let $K$ be a commutative ring with $1$. All $K$-algebras $A$ are
required to be commutative, have a $1$, and the homomorphism $K \to A$
is required to send $1$ to $1$. Homomorphisms of $K$-algebras are
unital ring homomorphisms $A \to B$ compatible with the homomorphisms from $K$
into them.

\begin{de}
An {\em $\FI$-algebra} $A$ over $K$ is a covariant functor from $\FI$ to the
category of algebras over $K$. Dually, $A$ gives rise to a contravariant
functor $X$ from $\FI$ to the category of affine schemes over $K$.
We call such a functor an {\em $\FIop$-scheme over $K$}.
\end{de}

A morphism $A \to B$ of $\FI$-algebras over $K$ is a natural
transformation of functors from $A$ to $B$: it consists of a $K$-algebra
homomorphism $\phi_S:A(S) \to B(S)$ for all finite sets $S$ such that for
all $S,T \in \FI$ and for each injection $\pi$ from $S$ to $T$, the following diagram commutes.
\[
\xymatrix{
A(S) \ar[r]^{\phi_S} \ar[d]_{A(\pi)} & B(S) \ar[d]^{B(\pi)} \\
A(T) \ar[r]_{\phi_T} & B(T).
}
\]
Morphisms of $\FIop$-schemes are defined dually. We often leave out the
subscript $S$ in $\phi_S$, and often just write $\pi$ instead of $A(\phi)$
when $A$ is clear from the context.

\subsection{$\FI$-algebras finitely generated in width $\leq d$}

$\FI$-algebras that are finitely generated in width at most $d$ are the main characters of this paper. By a subalgebra $B$ of an $\FI$-algebra $A$ over $K$ we mean an $\FI$-algebra $B$ over $K$ such that for each finite set $S$, $B(S)$ is a subalgebra of $A(S)$, and for each injection $\sigma : S \rightarrow T $, $B(\sigma)$ is the restriction of $A(\sigma)$ to $B(S)$.

\begin{de}
An $\FI$-algebra $A$ over $K$ is said to be {\em finitely
generated} if there is a finite set $D \subset \bigsqcup _{S\in \FI} A(S)$ which is not contained in any proper subalgebra of $A$. If such $D$ exists, we say that $A$ is generated by $D$.
\end{de}

\begin{de}
Let $A$ be an $\FI$-algebra over $K$. For  $S \in \FI$ and
$f \in A(S)$, we call the minimal $n$ such that $f$ lies in
$A(\pi) A([n])$ for some injective map $\pi: [n]\rightarrow
S$, the {\em width} of $f$, denoted $w(f)$.
\end{de}

\begin{de}
An $\FI$-algebra $A$ over $K$ is  {\em generated in width
$\leq d$} if there exists a collection $(f_i \in A(S_i))_{i \in I}$ of elements of width $\leq d$ that generates $A$.
\end{de}

\begin{de}
An $\FI$-algebras $A$ over $K$ is said to be finitely generated in width $\leq
d$ if $A$ is finitely generated and generated in width $\leq d$.
\end{de}

\begin{re}
It is straightforward to see that this is equivalent to the condition that
$A$ is generated by a finite collection of elements $f_i \in A(S_i)$
of width $\leq d$.
\end{re}

The following $\FI$-algebra $B_d$ is a building block for $\FI$-algebras that are finitely generated in width $\leq d$.
\begin{defi}
For a non-negative integer $d$, denote by $B_d$ the $\FI$-algebra over a ring $K$ that maps a finite set $S$ to a $K$-algebra
\begin{align*}
    B_d(S)=K[y_{i_1,i_2,...,i_d}:i_1,i_2,...,i_d \in S\,\, \text{and}\,\, i_j \neq i_l \,\, \text{when}\,\, j \neq l ]
\end{align*}
and  for an injection $\sigma : S \rightarrow T$, the
$K$-algebra homomorphism $B_d(\sigma)$ is determined by
$B_d(\sigma)(y_{i_1,i_2,...,i_d}):=
y_{\sigma(i_1),\sigma(i_2),...,\sigma(i_d)}$. It is easy to
see that $B_d$ is generated by the single element
$y_{1,2,...,d} \in B_d([d])$, of width $d$.
\end{defi}

\begin{re}\label{re:B[K_0,...,k_e]}
\begin{enumerate}
    \item We think of the elements of $\Spec B_d$ as $S\times S \times ... \times S$-tensors of which only the entries outside the big diagonal $\{ (i_1,i_2,...,i_d) : \exists \,\, j < l, i_j=i_l \}$ are given. We call these tensors ``off-diagonal tensors".
    \item It is easy to show that the tensor product of such algebras is again finitely generated in width at most the largest of the relevant $d$.
    \item For fixed non-negative integers $d, k_0,
    k_1,k_2,...,k_d$, $B[k_0, k_1, k_2,..., k_d]$ denotes
    the $\FI$-algebra $\bigotimes_{i=0}^d B_{i}^{\otimes
    k_i}$, which is finitely generated in width at most $d$.
    \qedhere
\end{enumerate}
\end{re}

\begin{de}
An $\FIop$-scheme of finite type over $K$ is the spectrum of a finitely generated $\FI$-algebra over $K$.
\end{de}

\begin{de}
A width-$d$ $\FIop$-scheme over $K$ is the spectrum of an $\FI$-algebra over $K$ that is generated in width $\leq d$.
\end{de}

\begin{de}
A width-$d$ $\FIop$-scheme of finite type over $K$ is the
spectrum of an $\FI$-algebra over $K$ that is finitely
generated in width $\leq d$.
\end{de}

We record the following lemmas for the later use.

\begin{lem}\label{lem:homomorphism}
Let $A$ be an $\FI$-algebra. Then any homomorphism $\phi :B_d \rightarrow A$ of $\FI$-algebras is completely determined by the image of $y_{1,2,...,d}\in B_d([d])$. Moreover, $B_d$ has the following universal property: if $a \in A([d])$ is any element of $A$, then there exists a unique homomorphism $\phi: B_d \rightarrow A$ such that $\phi(y_{1,2,...,d})=a$.
\end{lem}

\begin{proof}
Let $y_{i_1,i_2,...,i_d}$ be a variable in $B_d(S)$ where
$i_1,i_2,...,i_d \in S$. Let $\sigma :[d] \rightarrow S$ be
the injection that sends $j$ to $i_j$. Then  we have
\[ \phi(y_{i_1,i_2,...,i_d})= \phi (B_d(\sigma)
y_{1,2,...,d})= A(\sigma) \phi(y_{1,2,...,d}). \]
Thus the image of $y_{1,2,...,d}$ completely determines the map $\phi$.

To prove the last statement, for $i_1,\ldots,i_d \in S$ distinct,
define $\phi(y_{i_1,i_2,...,i_d}):= A(\sigma)a$, where $\sigma:[d] \to S$
is the map that sends $j$ to $i_j$, and extend this to a
homomorphism $B_d(S) \to A(S)$ in the unique manner. A
straightforward computation shows that this does, indeed,
define an $\FI$-algebra homomorphism $B \to A$ with
$\phi(y_{1,\ldots,d})=a$.
\end{proof}

\begin{re}\label{re:universality}
Lemma \ref{lem:homomorphism} implies that the $\FI$-algebra $B[k_0, k_1,k_2,...,k_d]$ has the same universal property when $a$ is replaced by a tuple in $\prod_{e=0}^{d} A([e])^{k_e}$.
\end{re}

\begin{lem}\label{lem:surjectivity}
Let $A$ be an $\FI$-algebra that is finitely generated in width at most $d$. Then there exists a surjective homomorphism $\phi : B[k_0, k_1,k_2,...,k_d] \rightarrow A$ for some non-negative integers $k_0, k_1,k_2,...,k_d$.
\begin{proof}
Let $k_i$ be the number of generators of width $i$. Assume
$h_1^{(i)},h_2^{(i)},...,h_{k_i}^{(i)} \in A([i])$ are
generators of width $i$. Consider the unique $\FI$-algebra homomorphism $\phi_i : B_{i}^{\otimes k_i} \rightarrow A$ that sends
$y_{1,2,...,i}^{(l)} \in B_i([i])^{\otimes k_i}$ to
$h_l^{(i)}$ for $l =1,2,...,k_i$. Then $\phi$ is surjective
by construction. Then the tensor map $\bigotimes_{i=0}^d: B[k_0, k_1,k_2,...,k_d] \rightarrow A$ is the required surjective morphism.
\end{proof}
\end{lem}

\begin{re}
Dually, $\Spec(A)$ embeds in a finite product of affine spaces of off-diagonal tensors.
\end{re}

\subsection{Shifting}
The idea of shifing an $\FI$-structure over a finite set goes back to
\cite{church-ellenberg-farb:FImodules}. It has also been used effectively
in \cite{draisma-eggermont-farooq:components}.

\begin{de}
Let $T_0$ be a finite set. Then $\Sh_{T_0}:\FI \to \FI$ is the functor that sends a finite set  $T$ to the disjoint union $T_0 \sqcup T$ and an injection $\sigma : T \rightarrow S$ to $\Sh_{T_0}
\sigma:T_0 \sqcup T \to T_0 \sqcup S$ that is the identity on $T_0$ and equal to $\sigma$ on $T$. For an $\FI$-algebra $A$ over $K$ we write $\Sh_{T_0}
A:=A \circ \Sh_{T_0}$ and for the $\FIop$-scheme $X=\Spec(A)$ over $K$ we write $\Sh_{T_0} X:=X \circ \Sh_{T_0} = \Spec(\Sh_{T_0}
A)$.
\end{de}

\begin{re}
If $A$ is finitely generated in width at most $d$, then so is $\Sh_{T_0}
A$; and hence, if $X$ is a width $\leq d$ $\FIop$-scheme of finite type over $K$, then so is $\Sh_{T_0} X$.

Furthermore, for a homomorphism $\phi:A \to B$ of $\FI$-algebras
over $K$, we write $\Sh_{T_0} \phi$ for the morphism $\Sh_{T_0} A \to
\Sh_{T_0} B$ that sends $T$ to $\phi(T_0 \sqcup T)$, and similarly for morphisms of $\FIop$-schemes. A straightforward check shows that $\Sh_{T_0}$ is a covariant functor from $\FI$-algebras over $K$
into itself and from $\FIop$-schemes over $K$ into itself.
\end{re}

\begin{lem}
For a fixed finite set $T_0$, $\Sh_{T_0}B[k_0, k_1, k_2,..., k_d]$ is isomorphic to $ B[k'_0, k'_1, k'_2,..., k'_{d-1},k_d]$ for some non-negative integers $k_0', k_1',...,k_{d-1}'$.
\begin{proof}
We have $\Sh_{T_0}(A \otimes C) \cong (\Sh_{T_0}A)\otimes (\Sh_{T_0}C)$, for any $\FI$-algebras $A$ and $C$. So it suffices to show that $B_d$ is isomorphic to $B[k_0',k_1',...,k_{d-1}',1]$. Note that
\begin{align*}
    \Sh_{T_0} B_d (T)=K[y_{i_1,i_2,...,i_d}:i_1,i_2,...,i_d
    \in T_0 \sqcup T\,\, \text{and}\,\, i_l \neq i_m \,\,
    \text{when}\,\, l \neq m ]\\ \cong
\left(\bigotimes_{e=0}^{d-1} B_e(T)^{\otimes \binom{d}{e}
\times |T_0| \times (|T_0|-1)\times ... \times
(|T_0|-d+e+1)} \right) \otimes (B_d(T) ),
\end{align*}
where the exponent counts the number of ways of filling $d-e$
of the $d$ index positions with elements of $T_0$.
Thus $\Sh_{T_0}B_d \cong B[k'_0, k'_1, k'_2,..., k'_{d-1},1]$.
\end{proof}
\end{lem}

\section{Noetherianity} \label{sec:Noetherianity}

\subsection{Noetherianity for $\FI$-algebras and
$\FIop$-schemes}

\begin{de}
Given an $\FI$-algebra $A$ over $K$, we define its direct limit
$\overrightarrow{A}= \underrightarrow{\lim}\,A([n])$
using the direct system $(A([n]), A(i_{m,n}))$, where $m \leq n$ and
$i_{m,n}: [m] \rightarrow [n]$ is the natural inclusion map.\\
Dually, given an $\FIop$-scheme $X$ over $K$, we define its  inverse limit
$\overleftarrow{X}= \underleftarrow{\lim} \,X([n])$ using the inverse system
$(X([n]), X(i_{m,n}))$, where $m \leq n$ and $i_{m,n}: [m] \rightarrow [n]$ is the natural inclusion map.
\end{de}

If $A$ be an $\FI$-algebra, then there is an induced action
of the group $\Sym(\infty)= \bigcup_{n\geq 1} \Sym(n)$, where $\Sym(n)$ denotes the symmetric group on $n$ letters and is naturally embedded into $\Sym(n+1)$ as the stabilizer of $\{n+1\}$, on the direct limit $\overrightarrow{A}$.
An $\FI$-algebra is called Noetherian if it satisfies the ascending chain condition on $\FI$-ideals. It is well known that if an $\FI$-algebra $A$ is Noetherian then its direct limit $\overrightarrow{A}$ is $\Sym(\infty)$-Noetherian, i.e.,it satisfies the ascending chain condition on $\Sym(\infty)$ stable ideals.\\
Duall, an $\FIop$-scheme $X$ is called Noetherian if it
satisfies the descending chain condition on closed
$\FI$-subschemes, and {\em topologically Noetherian} if it
satisfies the descending chain condition on closed and {\em reduced}
$\FI$-subschemes. Again the inverse limit
$\overleftarrow{X}$ of a topologically Noetherian
$\FIop$-scheme $X$ is topologically
$\Sym(\infty)$-Noetherian i.e.~it satisfies the descending
chain condition on $\Sym(\infty)$ stable closed reduced
subschemes. To save words, from now on, we will just use the word ``closed subset'' to mean
``reduced subscheme''.

\begin{ex}
The direct limit of the $\FI$-algebra $B_1^{\otimes k}$ is (isomorphic with) the polynomial ring $K[x_{i,j}: i \in [k], j \in \NN ]$ with the following induced action: $\sigma(x_{i,j})=x_{i,\sigma(j)}$ where $\sigma \in \Sym(\infty)$.\\
The inverse limit of the $\FIop$-scheme $\Spec B_1^{\otimes k}$ is the space of $k \times \NN$-matrices where $\Sym(\infty)$ acts by permuting columns of such matrices.
\end{ex}

\begin{thm}\cite{aschenbrenner-hillar:finitesymmetric, cohen:metabelian, cohen:closure, hillar-sullivant:GBinfdim} \label{thm:Noetherianity}
Every $\FI$-algebra $A$ over a Noetherian ring $K$ that is finitely generated in width $\leq 1$ is Noetherian, i.e., if $I_1 \subseteq I_2 \subseteq \cdots$ is an ascending chain of ideals in $A$, then $I_n=I_{n+1}$ for all $n \gg 0$.
\end{thm}

\begin{re}\label{re:SpecNoetherianity}
In fact, $A$ is Noetherian in a stronger sense: any finitely generated $A$-module satisfies the ascending chain condition on submodules \cite[Theorem 6.15]{Nagel-Romer}.\\
It is easy to see that if an $\FI$-algebra $A$ is Noetherian, then the $\FIop$-scheme $X$ defined by $A$ is topologically Noetherian.
\end{re}

\begin{ex}\label{ex:Noetherianity}
The $\FI$-algebra $B_1^{\otimes k}$ over a Noetherian ring $K$ is Noetherian because it is finitely generated in width at most 1. So the polynomial ring $K[x_{i,j}: i \in [k], j \in \NN ]$ is $\Sym(\infty)$-Noetherian.\\
Dually, the $\FIop$-scheme $\Spec B_1^{\otimes k}$ over a Noetherian ring $K$ is topologically Noetherian. Thus its inverse limit, the space of $k \times \NN$-matrices, is $\Sym(\infty)$-Noetherian.
\end{ex}

For the (easy) proofs of the following lemmas we refer to \cite[Section 5]{draisma-kutter:boundedtensors}.

\begin{lem}\label{lem:subspace}
Let $X$ be an $\FIop$-scheme that is topologically Noetherian. Then any closed subset of $X$ is topologically Noetherian with respect to the induced topology.
\end{lem}

\begin{lem}\label{lem:union}
Let $X$ and $Y$ be topologically Noetherian $\FIop$-schemes. Then the disjoint union $X \sqcup Y$ defined by $X \sqcup Y (S) = X(S) \sqcup Y(S)$ is topologically Noetherian with respect to the disjoint union topology.
\end{lem}

\begin{lem}\label{lem:image}
Let $\psi$ be a homomorphism of $\FIop$-schemes from $X$ to $Y$ and $X$ is topologically Noetherian. Then $\Im(\psi)$ is topologically Noetherian with respect to the topology induced from $Y$.
\end{lem}

\section{Reduction to the Case of Free FI-algebras}\label{sec:free}

This section is devoted to the reduction of our main theorem to the case of free $\FI$-algebras.

\subsection{Making $B$ free}
Here we introduce the following useful notations. Let $A$ be
an $\FI$-algebra. By $f\in A$ we mean that there exists $S
\in \FI$ such that $f \in A(S)$. Similarly, by $p \in \Spec
A $ we mean that there exists $T \in \FI$ such that $p \in
\Spec A(T)$. By $f(p)=0 $ we mean that for all $\sigma: S
\rightarrow T$ injections, $(\sigma f)(p)=0$.

\begin{prop}\label{prop:reductionB}
Let $\pi : C \rightarrow B$ and $\phi :B \rightarrow A$ be homomorphisms of $FI$-algebras, where $\pi$ is surjective. If the image closure $\closure {\Im (\pi^* \circ \phi^*)}$ of $\pi^* \circ \phi^*: \Spec A \rightarrow \Spec C$ is defined by finitely many elements in $C$, then the image closure $\closure{\Im(\phi^*)}$ of $\phi^*: \Spec A \rightarrow \Spec B$ is also defined by finitely many elements in $B$. If, moreover, $\closure {\Im (\pi^* \circ \phi^*)}$ is topologically Noetherian, then $\closure{\Im(\phi^*)}$ is also topologically Noetherian.
\begin{proof}
Suppose $f_1, f_2,...,f_m \in C$ define the image closure $\closure {\Im (\pi^* \circ \phi^*)}$. We claim that their images $\pi f_1, \pi f_2,...,\pi f_m \in B$ define the image closure $\closure{\Im(\phi^*)}$. First note that $\closure {\Im (\pi^* \circ \phi^*)}= \pi ^* (\closure{\Im (\phi^*))}$ as $\pi^*$ is a closed embedding. Let $p \in \closure{\Im (\phi^*)}$, then $\pi f_i(p)=f_i(\pi^*p)=0$. Thus for all $i=1,2,...,m$ we have that $\pi f_i$  vanishes on $\closure{\Im (\phi^*)}$.

Conversely, pick $p \in \Spec B$ such that $p$ is in the
closed subset defined by $\pi f_1,\pi f_2,...,\pi f_m$, then $\pi ^* p $ lies in the variety defined by $f_1,f_2,...,f_m$, that is, by our supposition, $\pi ^* p \in \closure {\Im (\pi^* \circ \phi^*)}$. Now let $g$ be an element of $\Ker(\phi)$. By the surjectivity of $\pi$, there exists $g'\in C$ such that $\pi g'=g$, which implies that $g'\in \Ker(\phi \circ \pi)$. Hence $g'$ vanishes on $\pi ^*(p)$. So $g(p)=\pi g'(p)=g'(\pi ^*(p))=0$. Hence $p \in \closure{\Im (\phi^*)}$.

The last statement is straight forward. Every descending
chain of closed subsets in $\closure{\Im (\phi^*)}$ maps
under $\pi ^*$ to a descending chain of closed subsets in
$\closure {\Im (\pi^* \circ \phi^*)}$ because $\pi ^*$ is a
closed embedding. The latter chain stabilizes and, and when
it does, so does the former chain. This completes the proof.
\end{proof}
\end{prop}

\begin{re}\label{re:B_is_free}
Let $B$ be an $\FI$-algebra that is finitely generated in
width at most $d$ and $\phi:B \rightarrow A$ be a homorphism
of $\FI$-algebras. Then, by Lemma \ref{lem:surjectivity}, there exists a surjective map $\pi$ from $B[k_0,k_1,..., k_d]$ to $B$. Thus Proposition \ref{prop:reductionB} implies that it suffices to prove the main theorem for the $\FI$-algebra $B[k_0,k_1,..., k_d]$ instead of $B$, i.e. without loss of generality, in the main theorem, we can assume that $B$ is free.
\end{re}

\subsection{Making $A$ free, as well}

\begin{prop}\label{prop:reductionA}
Let $\phi : B \rightarrow A$, $\psi : B \rightarrow A'$ and $\pi : A'\rightarrow A$ be homomorphisms of finitely generated $\FI$-algebras such that $\pi \circ \psi = \phi$. If the closure $\closure{\Im(\psi ^*)}$ of the image of $\psi ^*$ is topologically Noetherian, then the closure $\closure{\Im(\phi ^*)}$ of the image of $\phi ^*$ is also topologically Noetherian. If, moreover, $\closure{\Im(\psi ^*)}$ is defined by finitely many elements in $B$, then $\closure{\Im(\phi ^*)}$ is also defined by finitely many elements in $B$.
\begin{proof}
The first statement is a straight-forward implication of the fact that a closed subspace of a topologically Noetherian space is topologically Noetherian. Since $\Ker (\psi) $ is contained in $\Ker(\phi)$, we have
\begin{align*}
\closure{\Im(\phi ^*)}=V_{\Spec B}(\Ker (\phi)) = V_{\Spec B}(\Ker (\pi \circ \psi))    
\end{align*}
is a closed subset of $V_{\Spec B}(\Ker(\psi))= \closure{\Im(\psi ^*)}$. The latter is topologically Noetherian, hence so is the former.

For the last statement, observe that $\closure{\Im(\phi ^*)}$, being a closed subset of a topologically Noetherian space $\closure{\Im(\psi ^*)}$, is defined by finitely many elements in the coordinate ring $K[\closure{\Im(\psi ^*)}]=B/I$, where $I=I(\closure{\Im(\psi ^*)})$ is an ideal generated by finitely many elements in $B$. This implies that $\closure{\Im(\phi ^*)}$ is defined by finitely many elements in $B$.
\end{proof}
\end{prop}

\begin{re}\label{re:A_is_free}
Let $A$ be an $\FI$-algebra that is finitely generated in
width at most $1$ and let $\phi: B[k_0,k_1,..., k_d]
\rightarrow A$ be a morphism. Then by Lemma
\ref{lem:surjectivity} there is a surjective map $\pi
:B[k'_0,k'_1] \rightarrow A$. By universality of $B[k_0,k_1,..., k_d]$ (see
Remark \ref{re:universality}), there exists (not necessarily
unique) map $\psi : B[k_0,k_1,..., k_d] \rightarrow B[k'_0,k'_1]$ such that $\phi = \pi \circ \psi$. Thus Proposition \ref{prop:reductionA} implies that it suffices to prove the main theorem for $B[k'_0,k'_1]$ instead of $A$ i.e. without lost of generality, in the main theorem, we can assume that $A$ is free.
\end{re}

\section{Proof of the Main Theorem}

\subsection{Flattening}
For a more detailed description of flattening see \cite{draisma-kutter:boundedtensors}.
\begin{de}
Let $V= \bigotimes_{j=1}^n V_j$ be a tensor product of vector spaces $V_j$ over a field $L$. Then for each $i \in [n]$, there is a natural isomorphism
\begin{align*}
    \flat_{i}: V \rightarrow \left(\bigotimes_{j \neq i} V_j
    \right)  \otimes V_i.
\end{align*}
For a $t \in V$ the image $\flat_{i} (t)$ is a $2$-tensor
called a {\em flattening} of $t$.
\end{de}

A tensor has rank $1$ if it can be written as a tensor product of vectors $v_i \in V_i$.
A tensor $t$ has rank $l$ if $l$ is the minimum number of rank 1 tensors that sum to $t$. For $2$-tensors, the tensor rank is equal to the ordinary matrix rank.

\begin{re}
Flattening does not increase the rank of a tensor, i.e. $rk(\flat_i(t)) \leq rk(t)$.
\end{re}

\subsection{Tensors}

Define an $\FI$-algebra $C_d$ over a ring $K$ that sends a finite set $S$ to the $K$-algebra,
\begin{align*}
    C_d(S)=K[y_{i_1,i_2,...,i_d}: i_1,i_2,...,i_d \in S],
\end{align*}
and an injection $\sigma: S \rightarrow T$ to $C_d(\sigma): C_d(S) \rightarrow C_d(T)$ maps $y_{i_1, i_2,...,i_d}$ to $y_{\sigma (i_1), \sigma (i_2),...,\sigma(i_d)}$.

\begin{re}
For a finite set $S$, $\Spec C_d(S)$ is the set of $S \times S \times...\times S$-tensors, now including entries on the big diagonal.
\end{re}

There is a natural inclusion $\iota: B_d \rightarrow C_d$. Dually,
we have the projection map $\iota^*: \Spec C_d \rightarrow \Spec
B_d$. For a tensor $t \in \Spec C_d(S)$ we think of the
flattening $\flat_i(t)$ as a matrix whose rows and columns
are labeled by tuples in $S^{[d] \setminus \{i\}}$ and elements in $S$ respectively.

\begin{de}
A tuple $(x_i)_{i \in [d]}$ in $S^d$ is called a {\em distinct
value tuple} if $x_i \neq x_j $ when $i \neq j$. Denote by $DS^d$ the set of distinct value tuples in $S^d$.
\end{de}

\begin{de}
An {\em off-diagonal $l \times l$ sub-matrix} of $\flat_i(t)$ is a $u_1 \times u_2$-sub-matrix where $u_1 \subset DS^{d-1}$, $u_2 \subset S$ such that $u_1 \times u_2 \subset DS^d$ and $|u_1|=|u_2|=l$ .
By an off-diagonal $l \times l$-sub-determinant of $\flat _i(t)$, we mean the determinant of an off-diagonal $l \times l$-sub-matrix of $\flat_i(t)$ .
\end{de}

Fix non-negative integers $l$ and $d$. Let $Z_{d,l}$ be the subset of $\Spec B_d$ defined as follows: for all $S \in \FI$, $Z_{d,l}(S)$ is the variety defined by all off-diagonal $(l+1)\times (l+1)$-sub-determinants of $\flat _i(y),$ for all $i \in [d] $ and $y \in \Spec B_d(S)$. The following lemma shows that $Z_{d,l}$ is defined by finitely many equations in $B_d$.

\begin{lem}\label{lem:Z_l is f.gen.}
The subset $Z_{d,l}$ of $\Spec B_{d}$ is defined by finitely many equations in $B_{d}$.

\begin{proof}
  Let $S$ be a finite set, let $y \in \Spec B_{d}\left( S \right)$ and $i \in \left[ d \right]$ be arbitrary. Let $h$ be an arbitrary off-diagonal $\left( l+1 \right) \times \left( l+1 \right) $-subdeterminant of the flattening $\flat_{i}\left( y \right)$. The subdeterminant $h$ is given by indices $i_{1}, \dots, i_{l+1} \in S$ and distinct value tuples $\alpha_{1}, \dots, \alpha_{l+1} \in DS^{d-1}$ such that none of the indices $i_{j}$ occur in any of the $\alpha_{k}$ and such that all the $\alpha_{k}$ are distinct.\\
  Let $S'$ be the subset of $S$ consisting of all the
  entries of the $i_{j}$ and $\alpha_{k}$. Then $| S' | \leq
  \left( d-1 \right) \cdot \left( l+1 \right) + l+ 1 = d
  \cdot \left( l+1 \right)$. Set $ n = | S' | $ and choose a
  bijection $\rho \colon \left[ n \right] \rightarrow S'$.
  Define $\alpha'_{j} = \rho^{-1} \circ \alpha_{j}$ and
  $i'_{j} = \rho^{-1}(i_j)$ for all $j \in \left[ l+1
  \right]$, where we consider the $\alpha_{j}$ as maps from
  $\left[ d-1 \right]$ to $S'$.\\
  Then $h = B_{d}\left( \rho \right)h'$, where $h'$ is the off-diagonal subdeterminant of the flattening $\flat_{i}\left( \left( y_{\beta} \right)_{\beta \in D\left[ n \right]^{d-1}} \right)$ given by columns $i'_{1}, \dots, i'_{l+1}$ and rows $\alpha'_{1}, \dots, \alpha'_{l+1}$.
  Since $n$ can take on at most $d \cdot \left( l+1 \right)$ values, this construction yields finitely many polynomials $h'$. These polynomials define $Z_{d,l}$.
\end{proof}
\end{lem}

\begin{lem}\label{lem:containment_in_a_subdeterminant_variety}
Let $\phi : B_d \rightarrow B_{1}^{\otimes k}$ be an $\FI$-algebra homomorphism. Then $\closure {\Im (\phi^*)} \subset Z_{d,l}$ for some $l$.
\begin{proof}
By Lemma \ref{lem:homomorphism}, the map $\phi$ is determined by the image of $y_{1,2,...,d} \in B_d([d])$.
First we suppose that $y_{1,2,...,d} \in B_d([d])$ is mapped under $\phi$ to a monomial $M=\prod_{i\in [k],j \in [d]}x_{i,j}^{\alpha_{i,j}} \in B_1^{\otimes k}([d])$.
Dually, $a=(a_{i,j})_{i \in [k], j \in S} \in \Spec
B_1(S)^{\otimes k}$ is mapped to the projection (via
$\iota^*$) of the following tensor:
\begin{align*}
    (a_{j_1,j_2,...,j_d}= \prod_{i=1}^{k}a_{i,j_1}^{\alpha_{i,1}}\cdot \prod_{i=1}^{k}a_{i,j_2}^{\alpha_{i,2}} \cdots \prod_{i=1}^{k}a_{i,j_d}^{\alpha_{i,d}})_{j_1,j_2,...,j_d \in S}\\
    = \bigotimes_{j=1}^{d} (\prod_{i=1}^{k}a_{i,s}^{\alpha_{i,j}})_{s \in S} \in \Spec B_d(S),
\end{align*}

that is, $\phi^*(a)$ is the projection of a rank $\leq 1$ tensor. Thus $\Im(\phi^*)$ is contained in the projection of rank $\leq 1$ tensors.

Now, if $y_{1,2,...d}\in B_d([d])$ is mapped under $\phi$ to
$f = \sum_{c=1}^{l}M_c$ which is a sum of, say $l$,
monomials in the variables $x_{i,j}$ where $1 \leq i \leq k$
and $ 1\leq j \leq d$, then the subadditivity of tensor rank
implies that $\Im(\phi^*)$ lies in the projection $i^*$ of
tensors having rank at most $l$. Since flattening does not
increase rank, this implies $\Im(\phi^*) \subset Z_{d,l}$.
Hence $\closure {\Im (\phi^*)} \subset Z_{d,l}$.
\end{proof}
\end{lem}

\begin{re}
Let $\phi : B[k_0,k_1,...,k_d] \rightarrow B_{1}^{\otimes k}$ be an $\FI$-algebra homomorphism. By universality of $B[k_0,k_1,...,k_d]$,  $ \phi = \bigotimes_{e=0}^{d} (\bigotimes_{i=1}^{k_e}\phi_{e,i})$ where, $ \phi_{e,i} : B_e \rightarrow B_1^{\otimes k}$ for all $e$ and for all $i$,  and $\Im(\phi_{e,i}^*) \subset Z_{e,l_{e,i}}$ for some $l_{e,i}$. Let $l$ be the maximum of the set $\{l_{e,i}: 0 \leq e \leq d\,\, \text{and}\,\, 0 \leq i \leq k_e \}$. Then $\Im (\phi^*) \subset \prod _{e=0}^{d}Z_{e,l}^{k_e}$.
\end{re}

For a commutative ring $R$ with $1$, and a prime ideal $p \in \Spec R$, $\kappa(p)$ denotes the fraction field of the integral domain $R / p$.

\begin{lem}\label{lem:Y}
Let $Y$ be a subset of $\Spec C_d$. Suppose that there exists $N \in \ZZ_{\geq 0}$ such that for all $S $, for all $p \in Y(S)$, the tensor rank of $p$ over the field $\kappa(p)$ is at most $N$. Then there exist an $\FI$-algebra $A$ that is finitely generated in width at most $1$ and a map $\phi_N^*: \Spec A \rightarrow \Spec C_d $ such that $Y \subset \Im(\phi_N^*)$.
\end{lem}

In this lemma, $Y(S)$ is an arbitrary subset of $\Spec(C(S))$, in such
a manner that for any $\pi \in \Hom_\FI(S,T)$, the map $C(\pi)^*$ maps
$Y(T)$ into $Y(S)$.

\begin{proof}
Define a homomorphism of $\FI$-algebras $\phi_N : C_d
\rightarrow A=C_1^{\otimes [N]\times [d]}$ by sending
$y_{i_1,i_2,...,i_d}$ to $\sum_{j=1}^{N}
x_{j,1,i_1}x_{j,2,i_2}\cdots x_{j,d,i_d}$. By
assumption, for each $S$ and for every
$p=(p_{i_1,i_2,...,i_d})_{i_1,i_2,...,i_d \in S} \in Y(S)$,
there exists $(q_{j,l,i_l})_{(j,l,i_l)\in [N]\times
[d]\times S } \in \Spec C_1^{\otimes [N]\times
[d]}(S)(\kappa (p))$ such that $p_{i_1,i_2,...,i_d}=
\sum_{j=1}^{N} q_{j,1,i_1}q_{j,2,i_2}\cdots q_{j,d,i_d}$.
Hence $Y \subset \Im(\phi_N^*)$.
\end{proof}

\begin{lem}\label{lem:Z}
Let $Z$ be a subset of $\Spec B_d$. Suppose that there
exists $N \in \ZZ_{\geq 0}$ such that for all $S \in \FI$
and for all $p\in Z(S)$ there exists $\tilde{p} \in \Spec
C_d(S)$ such that the tensor rank of $\tilde{p}$ over the
field $\kappa(\tilde{p})$ is at most $N$ and $\iota^*(\tilde
p)=p$ where $\iota: B_d \rightarrow C_d$ is the inclusion map. Then there exists a map $\psi : B_d \rightarrow A$ such that $Z \subset \Im(\psi^*)$, where $A$ is an $\FI$-algebra that is finitely generated in width at most 1.

\begin{re}
The condition on the off-diagonal tensor $p$ is that it can be completed to a (full) tensor $\tilde{p}$ of rank $\leq N$, where $N$ does not depend on $p$ or $S$.
\end{re}
\begin{proof}
Define $Y:= \{ \tilde p \in (i^*)^{-1}(Z) :\,\, \text{tensor rank of}\,\, \tilde p \,\, \text{over}\,\, \kappa (\tilde p) \leq N \}$. By Lemma $\ref{lem:Y}$, $Y \subset \Im(\phi_N^*)$ and by construction, $Z= i^*(Y)$. Then $\psi = \phi_N \circ i $ is the required map such that $Z \subset \Im(\psi^*)$.
\end{proof}
\end{lem}

\begin{prop}\label{prop:induction}
Let $Z_{d,l}$ be the subset of $\Spec B_d$ defined above, by
the vanishing of off-diagonal $(l+1) \times
(l+1)$-subdeterminants of flattenings. Then there exists $N
\in \ZZ_{\geq 0}$ such that for all $S \in \FI$ and for all
$p\in Z_{d,l}(S)$ there exists $\tilde{p} \in \Spec C_d(S)$
such that the tensor rank of $\tilde{p}$ over the field
$\kappa(\tilde{p})$ is at most $N$ and $\iota^*(\tilde p)=p$
where $\iota: B_d \rightarrow C_d$ is the inclusion map.
\begin{proof}

We will prove this by a double induction: an outer induction
on $d$ and an inner induction on $l$. For $d=0$ and $d=1$,
we have $B_d=C_d$ and every element of $\Spec C_d(S)$ has
tensor rank at most 1, regardless of $l$. Assume the
proposition is true for $\leq d-1$. When $l=0$ note that
$Z_{d,0} $ contains only a zero tensor. Take $N=0$ and we
are done. Fix $l > 0$ and assume the proposition is true for
all smaller values. Write
$Z_{d,l}(S)=Z_{d,l-1}(S) \cup (Z_{d,l}(S) \setminus
Z_{d,l-1}(S))$. By the inner induction hypothesis, there exists $N_0$ such that for every $p\in Z_{d,l-1}(S)$ there exists $\tilde{p} \in \Spec C_d(S)$ such that the tensor rank of $\tilde{p}$ over the field $\kappa(\tilde{p})$ is at most $N_0$ and $i^*(\tilde p)=p$.

Now take $p \in Z_{d,l}([n]) \setminus Z_{d,l-1}([n])$. This
implies that there exist $i_0 \in [d]$ such that there is an
$l \times l $-sub-determinant $h$ of $\flat_{i_0}(p)$ which
is non-zero. The sub-determinant $h$ involves at most $\leq
l(d-1)+l=l \cdot d$ indices. Split $[n]$ into $S\sqcup T$
such that $S$ consists of indices that appear in $h$ and
$T=[n] \setminus S$; i.e., T is the set of indices that do
not appear in $h$. For each fixed subset $I \subset [d]$ and
fixed distinct value tuple $\alpha \in DS^{[d]\setminus I}$,
consider the off-diagonal tensor $p_{\alpha} = (p_{\alpha,
\beta})_{\beta \in DT^{I}}$. We regard $p_{\alpha}$ as an element of $\Spec B_{|I|}(T)$.

Case 1: If $I \neq [d]$, then the outer induction hypothesis
applies to $p_\alpha$. Indeed, flattening $\flat
_i(p_\alpha)$ for $i \in I$ yields a sub-matrix of
$\flat_i(p)$ and by the induction hypothesis, $p_\alpha = i^*(q_\alpha)$, where $q_\alpha \in \Spec C_{|I|}(T)$ has rank $\leq N_\alpha$.

Case 2: If $I=[d]$, then we want to find $q_{\emptyset} \in \Spec C_d (T)$ such that $p_{\emptyset}= i^*(q_{\emptyset})$ and the tensor rank of $q_{\emptyset}$ over the field $\kappa(p)$ is at most some $N_{\emptyset}$. Then we will have $p=i^*( \sum_{I,\alpha}(q_{\alpha}\,\, \text{padded with zeros}))$ such that the tensor rank of $p$ over the field $\kappa(p)$ is at most $N_1=\sum_{\alpha}N_{\alpha}$. Without loss of generality, assume that $i_0=d$. We have an off-diagonal $l \times l$-sub-matrix $M$ of the matrix $\flat_{d}(p)$ such that $\operatorname{det} (M)=h \neq 0$. Define $q_{\emptyset}$ as follows:

\begin{align*}
    q_{\emptyset} = \sum_{j \in u_2\,\, \text{and} \,\, \alpha \in u_1} (M^{-1})_{j,\alpha} q_{\{j\}} \otimes q_{\alpha}  = \sum _{j \in u_2} q_{\{j\}} \otimes (\sum_{\alpha \in u_1} (M^{-1})_{j,\alpha} \cdot q_{\alpha}).
\end{align*}

where $rk(q_{\{j\}}) \leq N_{\{j\}}$ and $rk (\sum_{\alpha \in u_1} (M^{-1})_{j,\alpha} \cdot q_{\alpha}) \leq 1$. Hence $rk(q_{\emptyset}) \leq N_{\emptyset} := \sum_{j \in u_2}N_{\{j\}}$.

Notations used in Figure \ref{fig:Flattening}: Rows and columns of the matrix are labeled by tuples in $(S \sqcup T)^{d-1}$ and elements in $S \sqcup T$ respectively. $M$ is a $u_1 \times u_2$-submatrix where $u_1 \subset DS^{d-1}$, $u_2 \subset S$ such that $u_1 \times u_2 \subset DS^d$ and $|u_1|=|u_2|=l$.

We want to show that $i^*(q_{\emptyset})=p_{\emptyset}$. Straightforward computations (from linear algebra) show that the matrix in the Figure \ref{fig:Flattening} has rank $=l$. In particular, all off diagonal $(l+1)\times (l+1)$-sub-determinants in positions $(u_1 \cup \{\alpha \})\times (u_2 \cup \{i\})$ are zero where $\alpha \in DT^{d-1}$ and $i \in T$. The same holds, by assumption for $p$. Hence, since $M$ is invertible and for all $j \in u_2$, $i^*(q_{\{j\}})=p_{\{j\}}$ and for all $\alpha \in u_1$, $i^*(q_{\alpha}) = p_{\alpha}$, we find that $(p_{\emptyset})_{\alpha,i} = (q_{\emptyset})_{\alpha,i}$, for all $\alpha \in DT^{d-1}$ and for all $i \in T$ such that $(\alpha,i) \in DT^{d}$. Hence $i^*(q_{\emptyset})=p_{\emptyset}$. $N=max\{N_0,N_1\}$ is the required $N$. This completes the proof.

\begin{figure}
    \centering
    \includegraphics[scale= 1.5]{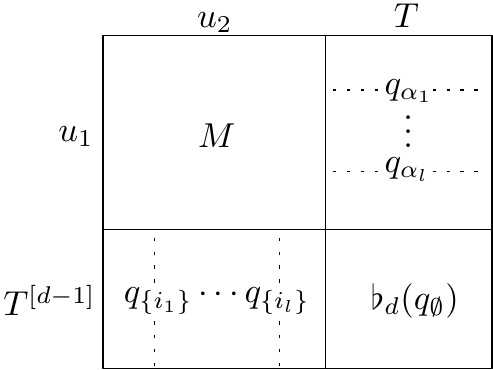}
    \caption{Flattening}
    \label{fig:Flattening}
\end{figure}
\end{proof}
\end{prop}

Recall from Remark \ref{re:B[K_0,...,k_e]} that for any nonnegative integers $e,k_0,k_1,...,k_{e}$,
\newline $B[k_0,k_1,...,k_{e}]$ denotes the following $\FI$-algebra $\bigotimes_{i=0}^e B_{i}^{\otimes k_i}$.
\begin{lem}\label{lem:Z over the tensor product}
For any morphism $\phi:B=B[k_0,k_1,...,k_{e}] \to A=B_1^{\otimes k}$ there exist a closed subset $Z \subset \Spec B$ defined by finitely many equations in $B$, an $\FI$-algebra $A'$ that is finitely generated in width at most $1$, and a morphism $\psi:B \to A'$ such that $\closure{\Im(\phi^*)} \subset Z \subset \Im(\psi^*)$.  
\end{lem}
\begin{proof}
We will prove this by an induction on the number of $\FI$-algebras appearing in the tensor product $B=B[k_0,k_1,...,k_{e}]$.

For the base case, suppose that $B=B_d$ for some nonnegative
integer $d$. Then by Lemma
\ref{lem:containment_in_a_subdeterminant_variety}, there
exists a positive integer $l$ such that
$\closure{\Im(\phi^*)}$ is a closed subset of $Z_{d,l}$. By
Lemma \ref{lem:Z_l is f.gen.}, $Z_{d,l}$ is defined by
finitely many equations in $B_d$. We take $Z=Z_{d,l}$. By
Proposition \ref{prop:induction} and Lemma \ref{lem:Z},
there exists a homomorphism $\psi : B_d \rightarrow A'$ of
$\FI$-algebras such that $Z=Z_{d,l}$ is contained in the image $\psi^*(\Spec A')$, where $A'$ is an $\FI$-algebra that is finitely generated in width at most 1.

Now suppose that $B=B_{d_1} \otimes B_{d_2} \otimes \cdots
\otimes B_{d_{n}}$. Here $d_i$'s are not necessarily
distinct. We write $B= C \otimes B_{d_n}$ where $C=B_{d_1}
\otimes B_{d_2} \otimes \cdots \otimes B_{d_{n-1}}$. We
think of $C$ and $B_{d_n}$ as subalgebras of $B$ via the
natural inclusions. Denote by $\phi_1$ and $\phi_2$ the
restrictions of the map $\phi: B \to A$ to $C$ and $B_{d_n}$
respectively. Then by the induction hypothesis, there exist
closed subsets $Z_1\subset \spec C$ and $Z_2\subset \Spec
B_{d_n}$ defined by finitely many equations in $C$ and in
$B_{d_n}$ respectively, $\FI$-algebras $A_1$ and $A_2$ both
are finitely generated in width at most $1$, and homomorphisms $\psi_1:C \to A_1$, $\psi_2:B_{d_n} \to A_2$ such that 
\begin{center}
$\closure{\Im(\phi_1^*)} \subset Z_1 \subset \Im(\psi_1^*)$ and $\closure{\Im(\phi_2^*)} \subset Z_2 \subset \Im(\psi_2^*)$.
\end{center}
We now have
\[ \overline{\Im(\phi^*)} \subseteq \overline{\Im(\phi_1^*)}
\times \overline{\Im(\phi_2^*)} \text{ and }
\Im(\psi_1^*)
\times \Im(\psi_2^*)=\Im((\psi_1 \otimes \psi_2)^*),\] 
where $\psi_1 \otimes \psi_2$ is the $\FI$-algebra homomorphism $C
\otimes B_{d_n} \to A_1\otimes A_2$ obtained as the tensor
product of $\psi_1$ and $\psi_2$. 

Note that the $\FI$-algebra $A_1 \otimes A_2$, being the tensor product of $\FI$-algebras that
are finitely generated in width at most $1$, is itself also finitely generated
in width at most $1$. Further, the product
$Z:=Z_1 \times Z_2$ is defined by union of the finite sets of equations
defining $Z_1$ and $Z_2$. Thus this $Z$, and $A':=A_1 \otimes
A_2$, and $\psi:=\psi_1\otimes \psi_2$ have the required
property $\closure{\Im(\phi^*)} \subset Z \subset \Im(\psi^*)$.
\end{proof}

\begin{proof}[Proof of the Main Theorem] Recall from Section \ref{sec:free} that without loss of generality we can assume that both $\FI$-algebras $B$ and $A$ are free, that is, $B=B[k_0,k_1,...,k_{e}]$ and $A=B_1^{\otimes k}$ for some nonnegative integers $k, k_0,k_1,...,k_e$.

From Lemma \ref{lem:Z over the tensor product}, $\closure{\Im(\phi^*)}$
is contained in a closed subset $Z \subset \Spec B$ which is defined
by finitely many equations in $B$, moreover, $Z$ is contained in
the image $\Im(\psi^*)$ of a width-one $\FIop$-scheme of finite type. Recall from
Remark \ref{re:SpecNoetherianity} that width-one $\FIop$-schemes of
finite type are topologically Noetherian and from Lemma \ref{lem:image}
that the image of a topologically Noetherian space is topologically
Noetherian. By Lemma \ref{lem:subspace}, the subset $Z$,
being a subset of the topologically Noetherian space
$\Im(\psi^*)$, is topologically
Noetherian. Thus
$\closure{\Im(\phi^*)}$ is topologically Noetherian being a closed
subset of $Z$ and defined by finitely many equations in the coordinate
ring of $Z$. Since $Z$ is defined by finitely many equations in $B$,
this implies that $\closure{\Im(\phi^*)}$ is defined by finitely many
equation in $B$. This completes the proof.
\end{proof}

\printbibliography
\end{document}